\documentclass{amsart}
\usepackage{url}
\usepackage{amssymb}

\newtheorem{thm}{Theorem}
\newtheorem{cor}{Corollary}
\newtheorem{conj}{Conjecture}
\newtheorem{lem}{Lemma}

\theoremstyle{definition}

\theoremstyle{remark}

\begin{document}

\title[{A short Proof of a conjecture by Hirschhorn and Sellers on Overpartitions }]
 {{A short Proof of a conjecture by Hirschhorn and Sellers on Overpartitions }}

\author{ LIUQUAN WANG}

\address{Department of Mathematics, National University of Singapore, Singapore, 119076, Singapore}

\email{mathlqwang@163.com}

\date{April 20, 2014}



\begin{abstract}
Let $\overline{p}(n)$ be the number of overpartitions of $n$,  we establish and give a short  elementary proof  of the following congruence
\[\overline{p}({{4}^{\alpha }}(40n+35))\equiv 0 \, (\bmod \, 40),\]
where $\alpha ,n $ are nonnegative integers. By letting $\alpha =0$ we proved a conjecture  of  Hirschhorn and Sellers. Some new congruences for $\overline{p}(n)$ modulo 3 and 5 have also been found,  including the following two infinite families of Ramanujan-type congruences: for any integers $n\ge 0$ and $\alpha \ge 1$,
\[\overline{p}({{5}^{2\alpha +1}}(5n+1))\equiv \overline{p}({{5}^{2\alpha +1}}(5n+4))\equiv 0 \, (\bmod  \, 5).\]

\end{abstract}

\maketitle

\section { Introduction and Main Results}
An overpartition of an integer $n$ is a partition wherein the first occurrence of a part may be overlined, and the number of overpartitions of $n$ is denoted by $\overline{p}(n)$. It's well known (see \cite{Lovejoy}, for example) that the generating function for $\overline{p}(n)$  is
\begin{equation}\label{gen}
\sum\limits_{n\ge 0}{\overline{p}(n){{q}^{n}}}=\frac{(-q;q)_{\infty}}{(q;q)_{\infty}}=\frac{1}{\varphi (-q)},
\end{equation}
where ${{(a;q)}_{\infty }}=(1-a)(1-aq)\cdots (1-a{{q}^{n}})\cdots $ is a standard $q$ series notation , and $\varphi (q)=\sum\nolimits_{n=-\infty }^{\infty }{{{q}^{{{n}^{2}}}}}$ is one of Ramanujan's Theta functions.

Overpartition was first introduced by MacMahon \cite{MacMahon} and has received much attention during the past ten years.  There're numerous results concerning
the arithmetic properties of $\overline{p}(n)$, for more information and references, we refer the reader to see \cite{Xia}, \cite{Lovejoy} and \cite{Sellers05} -\cite{Mahlburg}. Here we only mention some results which are related to our work.

In 2005, Hirschhorn and Sellers \cite{Hirschhorn05} gave many Ramanujan-type identities about $\overline{p}(n)$, for example
\begin{equation}\label{p4nid}
\sum\limits_{n\ge 0}{\overline{p}(4n+3)}=8\frac{{{({{q}^{2}};{{q}^{2}})}_{\infty }}({{q}^{4}};{{q}^{4}})_{\infty }^{6}}{(q;q)_{\infty }^{8}},
\end{equation}
which clearly implies $\overline{p}(4n+3)\equiv 0 \, (\bmod \,8)$. Meanwhile, they proposed a curious conjecture:
\begin{conj}\label{conj}
For any integer $n\ge 0$, we have
\[\overline{p}(40n+35)\equiv 0 \,(\bmod \, 40).\]
\end{conj}

In a recent paper, Chen and Xia \cite{Xia} give a proof of Conjecture \ref{conj} by employing 2-disections of quotients of theta functions and $(p,k)$-parametrization of theta functions. Their proof is relatively long and complicate, and our main goal is to give a  short proof. Indeed,  let ${{r}_{k}}(n)$ denotes the number of representations of $n$  as sum of $k$ squares, we find the following arithmetic relation:
\begin{thm}\label{thm1}
For any integer $n\ge 1$, we have
\[\overline{p}(5n)\equiv {{(-1)}^{n}}{{r}_{3}}(n) \, (\bmod \, 5).\]
\end{thm}

We have two remarkable corollaries.
\begin{cor}\label{thm1cor1}
For any integers $n\ge 0$ and $\alpha \ge 0$, we have
\[\overline{p}({{4}^{\alpha }}(40n+35))\equiv 0 \,(\bmod \, 5),\]
and
\[{{\overline{p}}}(5\cdot {{4}^{\alpha +1}}n)\equiv {{(-1)}^{n}}{{\overline{p}}}(5n)\, (\bmod \, 5).\]
\end{cor}
Let $\alpha=0$ in Corollary \ref{thm1cor1}, we get $\overline{p}(40n+35)\equiv 0 \,(\bmod \, 5)$, combining this with (\ref{p4nid}), Conjecture \ref{conj} follows immediately.

\begin{cor}\label{Tre}
For any prime $p\equiv -1 \, (\bmod \, 5)$, we have
\[\overline{p}(5{{p}^{3}}n)\equiv 0 \, (\bmod  \, 5)\]
 for all $n$ coprime to $p$.
\end{cor}

Corollary \ref{Tre} was first proved by  Treneer (see Proposition 1.4 in  \cite{Treneer}) in 2006 using the theory of modular forms, which is  not elementary.

Furthermore, with Corollary \ref{thm1cor1} in mind, we are able to generalize Conjecture \ref{conj} to the following
\begin{thm}\label{genthm}
For any integers $n\ge 0$ and $\alpha \ge 0$, we have
\[\overline{p}({{4}^{\alpha }}(40n+35))\equiv 0 \, (\bmod \, 40).\]
\end{thm}

Some miscellaneous congruences can be deduced from Theorem \ref{thm1}, we list some of them here.
\begin{thm}\label{5alpha}
For any integers $\alpha \ge 1$ and $n \ge 0$, we have
\[\overline{p}({{5}^{2\alpha +1}}(5n+1))\equiv \overline{p}({{5}^{2\alpha +1}}(5n+4))\equiv 0 \, (\bmod \, 5).\]
\end{thm}

\begin{thm}\label{2case}
Let $p\ge 3$ be a prime, $N$ a positive integer which is coprime to $p$. Let $\alpha $ be any nonnegative integer. \\
(1) If $p\equiv 1 \, (\bmod  \, 5)$, then $\overline{p}(5{{p}^{10\alpha +9}}N)\equiv 0 \, (\bmod  \, 5).$ \\
(2) If $p\equiv 2,3,4 \, (\bmod  \, 5)$, then $\overline{p}(5{{p}^{8\alpha +7}}N)\equiv 0 \, (\bmod  \, 5).$
\end{thm}

Finally, we mention that in 2011, based on the generating function of $\overline{p}(3n)$ discovered by Hirschhorn and Sellers \cite{Sellers05},  Lovejoy and Osburn \cite{Osburn} proved: for any integer $n\ge 1$, we have
\[\overline{p}(3n)\equiv {{(-1)}^{n}}{{r}_{5}}(n)\, (\bmod \, 3).\]
Using the same  method  from the proof of Theorem \ref{thm1}, we are able to improve this congruence to the following one.
\begin{thm} \label{Lov}
For any integer $n\ge 1$, we have
\[\overline{p}(3n)\equiv {{(-1)}^{n}}{{r}_{5}}(n)\, (\bmod \, 9).\]
\end{thm}

Similar to Theorem \ref{2case}, we can deduce the following interesting congruences from Theorem \ref{Lov}.
\begin{thm}\label{Lovcor}
 Let $p\ge 3$ be a prime and $N$ a positive integer which is coprime to $p$. \\
(1) If $p\equiv 1 \, (\bmod  \, 3)$, then $\overline{p}(3{{p}^{6\alpha +5}}N)\equiv 0 \, (\bmod  \, 3)$, and $\overline{p}(3{{p}^{18\alpha +17}}N)\equiv 0 \, (\bmod \, 9).$ \\
(2) If $p\equiv 2 \, (\bmod  \, 3)$, then $\overline{p}(3{{p}^{4\alpha +3}}N)\equiv 0 \, (\bmod  \, 9).$
\end{thm}

\section{Preliminaries}
\begin{lem}[cf. Lemma 1.2 in \cite{pod}]\label{palpha}
Let $p$ be a prime and $\alpha $ a positive integer. Then
\[(q;q)_{\infty }^{{{p}^{\alpha }}}\equiv ({{q}^{p}};{{q}^{p}})_{\infty }^{{{p}^{\alpha -1}}} \, (\bmod  \, {{p}^{\alpha }}).\]
\end{lem}

\begin{lem}[cf. Theorem 3.3.1,  Theorem 3.5.4 in \cite{Bruce}]\label{r48}
For any integer $n \ge 1$, we have
\[{{r}_{4}}(n)=8\sum\limits_{d|n,4 \nmid d}{d}, \, \,{{r}_{8}}(n)=16{{(-1)}^{n}}\sum\limits_{d|n}{{{(-1)}^{d}}{{d}^{3}}}.\]
\end{lem}

\begin{lem}\label{r48modp}
For any prime $p\ge 3$, we have
 \[{{r}_{4}}(pn)\equiv {{r}_{4}}(n)\, (\bmod \,p),{{r}_{8}}(pn)\equiv {{r}_{8}}(n)\, (\bmod \, p^3).\]
\end{lem}
\begin{proof}
By Lemma \ref{r48}, we have
\begin{displaymath}
 {{r}_{4}}(n)=8\sum\limits_{d|n,4 \nmid d}{d}=8\sum\limits_{\begin{smallmatrix}
 d|n \\
 4 \nmid d,p \nmid d
\end{smallmatrix}}{d}+8\sum\limits_{\begin{smallmatrix}
 d|n \\
 4 \nmid d, p | d
\end{smallmatrix}}{d}\equiv 8\sum\limits_{\begin{smallmatrix}
 d|n \\
 4 \nmid d,p \nmid d
\end{smallmatrix}}{d}\, (\bmod \, p),
\end{displaymath}
and
\begin{displaymath}
{{r}_{4}}(pn)=8\sum\limits_{d|pn,4 \nmid d}{d}=8\sum\limits_{\begin{smallmatrix}
 d|pn \\
 4\nmid d,p \nmid d
\end{smallmatrix}}{d}+8\sum\limits_{\begin{smallmatrix}
 d|pn \\
 4 \nmid d,p | d
\end{smallmatrix}}{d}=8\sum\limits_{\begin{smallmatrix}
 d|n \\
 4 \nmid d,p \nmid d
\end{smallmatrix}}{d}+8p\sum\limits_{\begin{smallmatrix}
 d|n \\
 4 \nmid d
\end{smallmatrix}}{d}.
\end{displaymath}
Comparing the two identities above,  we see that ${{r}_{4}}(pn)\equiv {{r}_{4}}(n)\, (\bmod \, p).$

Similarly we can prove ${{r}_{8}}(pn)\equiv {{r}_{8}}(n)\, (\bmod \, p^3).$
\end{proof}

\begin{lem}[cf. Theorem 1 in Chapter 4 of \cite{Grosswald}]\label{r30}
For any integers $\alpha  \ge 0$ and $ n\ge 0$, we have ${{r}_{3}}({{4}^{\alpha }}(8n+7))=0$ and ${{r}_{3}}({{4}^{\alpha }}n)={{r}_{3}}(n).$
\end{lem}

\begin{lem}[cf. \cite{Hsquare}]\label{r3relation}
Let $p \ge 3$ be a prime, for any  integers $n \ge 1$  and $\alpha \ge 0$,  we have
\[{{r}_{3}}({{p}^{2 \alpha }}n)=\Bigg(\frac{{{p}^{\alpha +1}}-1}{p-1}-\Big(\frac{-n}{p}\Big)\frac{{{p}^{\alpha }}-1}{p-1}\Bigg){{r}_{3}}(n)-p\frac{{{p}^{\alpha }}-1}{p-1}{{r}_{3}}(n/{{p}^{2}}).\]
where $(\frac{\cdot }{p})$ denotes the Legendre symbol,  and we take ${{r}_{3}}(n/{{p}^{2}})=0$ unless  ${{p}^{2}} | n$ .
\end{lem}

\begin{lem}[cf. Theorem 3 in \cite{Kim1}]\label{mod8}
Let $n$ be an integer which is neither a square nor twice a square, then $\overline{p}(n)\equiv 0 \, (\bmod \, 8)$.
\end{lem}

\begin{lem}[cf. \cite{Cooper}]\label{r5relation}
Let $p\ge 3$ be a prime, $n$ a positive integer and ${{p}^{2}} \nmid n$. For any integer $\alpha \ge 0$, we have
\[{{r}_{5}}({{p}^{2\alpha }}n)=\Bigg(\frac{{{p}^{3\alpha +3}}-1}{{{p}^{3}}-1}-p\Big(\frac{n}{p}\Big)\frac{{{p}^{3\alpha }}-1}{{{p}^{3}}-1}\Bigg){{r}_{5}}(n).\]
\end{lem}

\section{Proofs of The Theorems}

\begin{proof}[Proof of Theorem \ref{thm1}]
Replace $q$ by $-q$ in (\ref{gen}) we get
\[\sum\limits_{n\ge 0}{\overline{p}(n){{(-q)}^{n}}}=\frac{1}{\varphi (q)},\]
hence we have
\[\varphi {{(q)}^{5}}\sum\limits_{n\ge 0}{\overline{p}(n){{(-q)}^{n}}}=\varphi {{(q)}^{4}}=\sum\limits_{n\ge 0}{{{r}_{4}}(n){{q}^{n}}}.\]
By Lemma \ref{palpha}, we have $\varphi {{(q)}^{5}}\equiv \varphi ({{q}^{5}})\,(\bmod \,5)$ and thus
\[\varphi ({{q}^{5}})\sum\limits_{n\ge 0}{\overline{p}(n){{(-q)}^{n}}}\equiv \sum\limits_{n\ge 0}{{{r}_{4}}(n){{q}^{n}}}\,(\bmod \,5).\]
Collecting  all the terms of the form ${{q}^{5n}}$  on both sides, we get
\[\varphi ({{q}^{5}})\sum\limits_{n\ge 0}{\overline{p}(5n){{(-q)}^{5n}}}\equiv \sum\limits_{n\ge 0}{{{r}_{4}}(5n){{q}^{5n}}}\,(\bmod \,5),\]
replace ${{q}^{5}}$ by $q$  and apply Lemma \ref{r48modp} with $p=5$  we obtain
\[\varphi (q)\sum\limits_{n\ge 0}{\overline{p}(5n){{(-q)}^{n}}}\equiv \sum\limits_{n\ge 0}{{{r}_{4}}(5n){{q}^{n}}}\equiv \sum\limits_{n\ge 0}{{{r}_{4}}(n){{q}^{n}}}=\varphi {{(q)}^{4}}\,(\bmod \,5).\]
Hence we have
\[\sum\limits_{n\ge 0}{\overline{p}(5n){{(-q)}^{n}}}\equiv \varphi {{(q)}^{3}}=\sum\limits_{n\ge 0}{{{r}_{3}}(n){{q}^{n}}}\,(\bmod \,5),\]
and Theorem \ref{conj} follows by comparing the coefficients of ${{q}^{n}}$ on both sides.
\end{proof}

\begin{proof}[Proof of Corollary \ref{thm1cor1}]
 This corollary follows immediately by Theorem \ref{thm1} and Lemma \ref{r30}.
\end{proof}
\begin{proof}[Proof of Corollary \ref{Tre}]
By Theorem \ref{thm1}, we can replace ${{r}_{3}}(n)$ by ${{(-1)}^{n}}{{\overline{p}}}(5n)$ throughout  Lemma \ref{r3relation} with $\alpha =1$,  we deduce that
\[{{\overline{p}}}(5{{p}^{2}}m)\equiv \Big(p+1-(\frac{-m}{p})\Big){{\overline{p}}}(m)-p{{\overline{p}}}(m/{{p}^{2}}) \, (\bmod \, 5).\]
 Let $m=np$, then $\overline{p}(m/{{p}^{2}})=\overline{p}(n/p)=0$ since $n$ is coprime to $p$,  the theorem follows immediately.
\end{proof}

 \begin{proof}[Proof of Theorem \ref{genthm}]
Since $40n+35=5(8n+7)$ is an odd number, it can not be twice a square.  If $5(8n+7)={{x}^{2}}$ is a square where $x$ is an odd number, then we know $5|x$.  Let $x=5y$ where $y$ is an odd number, we get $8n+7=5{{y}^{2}}$, but $5{{y}^{2}}\equiv 5 \, (\bmod \, 8)$, this is a contradiction. Hence we know ${{4}^{\alpha }}(40n+35)$ is neither a square nor twice a square, by Lemma \ref{mod8} we have $\overline{p}({{4}^{\alpha }}(40n+35))\equiv 0 \, (\bmod \, 8)$, combining this with Corollary \ref{thm1cor1} we complete our proof.
\end{proof}

\begin{proof}[Proof of Theorem \ref{5alpha}]
Set $p=5$ in Lemma \ref{r3relation}, $n=5m+r,r\in \{1,4\}$, it's easy to deduce ${{r}_{3}}({{5}^{2\alpha }}(5m+r))\equiv 0 \, (\bmod \, 5)$ for any integer $\alpha \ge 1$. By Theorem \ref{thm1} we complete our proof.
\end{proof}

\begin{proof}[Proof of Theorem \ref{2case}]
(1)  Let $n=pN$ in Lemma \ref{r3relation}, then replace $\alpha $ by $5\alpha +4$, we have
\[\frac{{{p}^{5\alpha +5}}-1}{p-1}=1+p+\cdots +{{p}^{5\alpha +4}}\equiv 0 \, (\bmod  \, 5),\]
hence ${{r}_{3}}({{p}^{10\alpha +9}}N)\equiv 0  \, (\bmod  \, 5).$ By Theorem \ref{thm1} we deduce $\overline{p}(5{{p}^{10\alpha +9}}N)\equiv 0 \, (\bmod  \, 5).$ \\
(2)  Let $n=pN$ in Lemma \ref{r3relation}, then replace $\alpha $ by $4\alpha +3$, since  ${{p}^{4\alpha +4}}\equiv 1 \, (\bmod  \, 5)$, we deduce ${{r}_{3}}({{p}^{8\alpha +7}}N)\equiv 0 \, (\bmod  \, 5),$ by Theorem \ref{thm1} we deduce $\overline{p}(5{{p}^{8\alpha +7}}N)\equiv 0 \, (\bmod  \, 5).$
\end{proof}

 \begin{proof}[Proof of Theorem \ref{Lov}]
We have
\[\varphi {{(q)}^{9}}\sum\limits_{n\ge 0}{\overline{p}(n){{(-q)}^{n}}}=\varphi {{(q)}^{8}}=\sum\limits_{n\ge 0}{{{r}_{8}}(n){{q}^{n}}}.\]
Thanks to Lemma \ref{palpha}, we have $\varphi {{(q)}^{9}}\equiv \varphi {{({{q}^{3}})}^{3}}\,(\bmod \,  9)$ and thus
\[\varphi {{({{q}^{3}})}^{3}}\sum\limits_{n\ge 0}{\overline{p}(n){{(-q)}^{n}}}\equiv \sum\limits_{n\ge 0}{{{r}_{8}}(n){{q}^{n}}}\,(\bmod \,  9).\]
Collecting  all the terms of the form ${{q}^{3n}}$ on both sides, we get
\[\varphi {{({{q}^{3}})}^{3}}\sum\limits_{n\ge 0}{\overline{p}(3n){{(-q)}^{3n}}}\equiv \sum\limits_{n\ge 0}{{{r}_{8}}(3n){{q}^{3n}}}\,(\bmod \,  9),\]
replace ${{q}^{3}}$ by $q$  and apply Lemma \ref{r48modp} with $p=3$, we obtain
\[\varphi {{(q)}^{3}}\sum\limits_{n\ge 0}{\overline{p}(3n){{(-q)}^{n}}}\equiv \sum\limits_{n\ge 0}{{{r}_{8}}(3n){{q}^{n}}}\equiv \sum\limits_{n\ge 0}{{{r}_{8}}(n){{q}^{n}}}=\varphi {{(q)}^{8}}\,(\bmod \,  9).\]
Hence we have
\[\sum\limits_{n\ge 0}{\overline{p}(3n){{(-q)}^{n}}}\equiv \varphi {{(q)}^{5}}=\sum\limits_{n\ge 0}{{{r}_{5}}(n){{q}^{n}}}\,(\bmod \,  9),\]
and Theorem \ref{Lov} follows by comparing the coefficients of ${{q}^{n}}$ on both sides.
\end{proof}

\begin{proof}[Proof of Theorem \ref{Lovcor}]
(1)  Let $n=pN$ in Lemma \ref{r5relation}, then replace $\alpha $ by $3\alpha +2$, we have
\[\frac{{{p}^{9\alpha +9}}-1}{{{p}^{3}}-1}=1+{{p}^{3}}+\cdots +{{p}^{3(3\alpha +2)}}\equiv 0 \, (\bmod  \, 3),\]
hence ${{r}_{5}}({{p}^{6\alpha +5}}N)\equiv 0 \, (\bmod  \, 3).$ By Theorem \ref{Lov} we deduce $\overline{p}(3{{p}^{6\alpha +5}}N)\equiv 0 \, (\bmod  \, 3).$

Similarly, let $n=pN$ in Lemma \ref{r5relation} and replace $\alpha $ by $9\alpha +8$.  Since $p\equiv 1 \, (\bmod \, 3)$ implies ${{p}^{3}}\equiv 1 \, (\bmod  \, 9)$, we have
\[\frac{{{p}^{27\alpha +27}}-1}{{{p}^{3}}-1}=1+{{p}^{3}}+\cdots +{{p}}^{3(9\alpha +8)}\equiv 0 \, (\bmod  \, 9).\]
Hence ${{r}_{5}}({{p}^{18\alpha +17}}N)\equiv 0 \, (\bmod  \, 9)$, by Theorem \ref{Lov} we deduce $\overline{p}(3{{p}^{18\alpha +17}}N)\equiv 0  \, (\bmod   \, 9)$. \\
(2) Let   $n=pN$ in Lemma \ref{r5relation}, then replace $\alpha $ by $2\alpha +1$. Note that $p\equiv 2 \,(\bmod  \,3)$ implies  ${{p}^{3}}\equiv -1 \,(\bmod \,9)$. Since ${{p}^{6\alpha +6}}\equiv 1 \, (\bmod  \, 9)$, we have ${{r}_{5}}({{p}^{4\alpha +3}}N)\equiv 0 \, (\bmod  \, 9)$.  By Theorem  \ref{Lov} we deduce $\overline{p}(3{{p}^{4\alpha +3}}N)\equiv 0 \, (\bmod  \, 9).$
\end{proof}

\bibliographystyle{amsplain}

\end{document}